\documentclass[reqno,a4paper]{amsart}
\usepackage[utf8]{inputenc}
\usepackage[T1]{fontenc}
\usepackage[british]{babel}
\usepackage{amsfonts, amssymb, amsmath}
\usepackage[all]{xy}
\usepackage{hyperref}
\usepackage{mathbbol}
\usepackage{calrsfs}
\usepackage{mathabx}
\usepackage{amsfonts}
\usepackage{amssymb}
\usepackage{amsthm}
\usepackage{amsmath}
\allowdisplaybreaks
\usepackage{xcolor}
\definecolor{l-gray}{gray}{0.7}

\theoremstyle{plain}
\newtheorem{theorem}{Theorem}[subsection]
\newtheorem{proposition}[theorem]{Proposition}
\newtheorem{corollary}[theorem]{Corollary}
\newtheorem{lemma}[theorem]{Lemma}

\theoremstyle{definition}
\newtheorem{definition}[theorem]{Definition}

\theoremstyle{remark}
\newtheorem{remark}[theorem]{Remark}

\newtheorem{example}[theorem]{Example}

\numberwithin{equation}{section}
\numberwithin{figure}{section}

\newenvironment{tfae}
{
	\begin{enumerate}}
	{\end{enumerate}}

\newcommand{\noproof}{\hfill \qed}

\newcommand{\couniv}[2]{\langle #1, #2\rangle}
\newcommand{\counivtr}[3]{\langle #1, #2, #3\rangle}

\newcommand{\defn}{\textbf}

\newcommand{\cosmash}{\diamond}

\newcommand{\C}{\ensuremath{\mathcal{C}}}

\newcommand{\V}{\mathcal{V}}
\newcommand{\W}{\mathcal{W}}

\newcommand{\SSE}{\ensuremath{\mathsf{SSE}}}
\newcommand{\Pt}{\ensuremath{\mathsf{Pt}}}

\newcommand{\Cat}{\ensuremath{\mathsf{Cat}}}

\newcommand{\Gp}{\ensuremath{\mathsf{Gp}}} 

\newcommand{\XMod}{\ensuremath{\mathsf{XMod}}}

\newcommand{\Ens}{\mathsf{Set}}
\newcommand{\Set}{\mathsf{Set}}
\newcommand{\LieK}{\mathsf{Lie_{\mathbb{K}}}}

\renewcommand{\P}{{\rm (P)}}

\newcommand{\K}{\mathbb{K}}
\newcommand{\N}{\mathbb{N}}

\newcommand{\Z}{\mathbb{Z}}

\newcommand{\Left}{\mathrm{L}}

\DeclareMathOperator{\Coker}{Coker}

\newdir{>>}{{}*!/3.5pt/:(1,-.2)@^{>}*!/3.5pt/:(1,+.2)@_{>}*!/7pt/:(1,-.2)@^{>}*!/7pt/:(1,+.2)@_{>}} 
\newdir{ >>}{{}*!/8pt/@{|}*!/3.5pt/:(1,-.2)@^{>}*!/3.5pt/:(1,+.2)@_{>}}
\newdir{ |>}{{}*!/-3.5pt/@{|}*!/-8pt/:(1,-.2)@^{>}*!/-8pt/:(1,+.2)@_{>}} 
\newdir{ >}{{}*!/-8pt/@{>}}
\newdir{>}{{}*:(1,-.2)@^{>}*:(1,+.2)@_{>}}
\newdir{<}{{}*:(1,+.2)@^{<}*:(1,-.2)@_{<}}

\def\pullback{
	\ar@{-}[]+R+<6pt,-1pt>;[]+RD+<6pt,-6pt>%
	\ar@{-}[]+D+<1pt,-6pt>;[]+RD+<6pt,-6pt>}

\def\reversepullback{ 
	\ar@{-}[]+R+<6pt,-1pt>;[]+RU+<6pt,6pt>%
	\ar@{-}[]+U+<-1pt,6pt>;[]+RU+<6pt,6pt>}

\def\dottedpullback{%
	\ar@{.}[]+R+<6pt,-1pt>;[]+RD+<6pt,-6pt>%
	\ar@{.}[]+D+<1pt,-6pt>;[]+RD+<6pt,-6pt>}

\def\pushout{%
	\ar@{-}[]+L+<-6pt,1pt>;[]+LU+<-6pt,6pt>%
	\ar@{-}[]+U+<-1pt,6pt>;[]+LU+<-6pt,6pt>}

\usepackage{tikz-cd}

\DeclareMathOperator{\coker}{coker}

\usepackage{hyperref}
\hypersetup{
	colorlinks=true,       
	linkcolor=blue,          
	citecolor=magenta,        
	filecolor=green,      
	urlcolor=cyan           
}

\allowdisplaybreaks

\begin{document}
\emergencystretch 3em

\title[Projective crossed modules in semi-abelian categories]{Projective crossed modules\\in semi-abelian categories}

\author{Maxime Culot}

\address{Institut de Recherche en Math\'ematique et Physique, Universit\'e catholique de Louvain, che\-min du cyclotron~2 bte~L7.01.02, B--1348 Louvain-la-Neuve, Belgium}

\email{maxime.culot@uclouvain.be}

\thanks{Competing interests: The authors declare none. Author Contributions: All authors discussed the results and contributed to the final manuscript. Data Availability: No new data were created or analysed during this study. Data sharing is not applicable to this article. Funding declaration: The authors did not receive support from any organization for the submitted work.}

\begin{abstract}
	We characterize projective objects in the category of internal crossed modules within any semi-abelian category. When this category forms a variety of algebras, the internal crossed modules again constitute a semi-abelian variety, ensuring the existence of free objects, and thus of enough projectives. We show that such a variety is not necessarily Schreier, but satisfies the so-called Condition (P)---meaning the class of projectives is closed under protosplit subobjects---if and only if the base variety satisfies this condition. As a consequence, the non-additive left chain-derived functors of the connected components functor are well defined in this context.
\end{abstract}


\maketitle

\tableofcontents

\section{Introduction}

Projective objects are fundamental in homological algebra, particularly for defining derived functors of additive functors (see, e.g., \cite{Cartan-Eilenberg, MacLane:Homology, Weibel}). As they always have enough projectives, varieties of algebras provide a natural context for these objects. In some cases, such as all Schreier varieties, projectives coincide with free objects, which have explicit constructions (see for instance~\cite{Adamek-Rosicky-Vitale}). A~well-known example is the variety $\Gp$ of groups. However, intriguingly, the category of internal categories in $\Gp$—equivalently, the variety of crossed modules—does not share this property. In fact, \cite{Carrasco-Homology} constructs a projective crossed module that is not free.

In this paper, we explore the category of internal crossed modules $\XMod(\C)$ as introduced by G.~Janelidze~\cite{Janelidze}, for a given semi-abelian category $\C$. To streamline the presentation and simplify certain proofs, we use the alternative characterization of~\cite{HVdL} based on the theory of cosmash products (see Section~\ref{sec - cosmash products and XMOD}).

It is already known that if we replace $\C$ with a semi-abelian variety $\V$, then $\XMod(\V)$ itself forms a variety and thus, it has free objects and hence enough projectives. This framework enables us to extend the result of Carrasco, Cegarra et R.-Grandjeán~\cite{Carrasco-Homology} to any semi-abelian variety of algebras (see Corollary~\ref{cor - XMod(V) not Schreier}), which includes the case of $\Gp$. Our key findings rely on the construction of a projective internal crossed module (see Theorem~\ref{thm - projective in XMod(C)}) and the description of free internal crossed modules (see Corollary~\ref{cor - free object XMod(V)}). These results rely on preliminary findings related to the category of split extensions over a fixed object $B \in \C$, denoted by $\SSE_B(\C)$ (see Section~\ref{subsec-SSEB(V)}).

Another important consequence of Theorem~\ref{thm - projective in XMod(C)} is the study of the so-called Condition \P\ that the class of projectives is closed under protosplit subobjects (see Section~\ref{sec - condition P}). We show that a semi-abelian variety $\V$ satisfies \P\ if and only if $\XMod(\V)$ does (see Theorems~\ref{thm - proof of P XMod(V)} and~\ref{thm - P iff}). This establishes a new class of varieties satisfying \P. To do so, we only need to start from a semi-abelian variety $\V$ that itself satisfies \P. This includes trivial examples such as all abelian varieties and all Schreier varieties\footnote{There exist abelian varieties that are Schreier (e.g., the category of modules over a principal ideal domain), while others are not (e.g., the category of modules over the ring $\Z/4\Z$).}, and non-trivial examples such as the category of Lie algebras over a unital ring.

As first introduced in \cite{MC-FR-TVdL}, Condition \P\ is fundamental for defining the non-additive derived functors of a protoadditive functor that preserves coproducts and proper morphisms. In Section~\ref{sec - non additive functor pi0}, we show that non-additive derived functors (in the sense of Culot, Renaud, and Van der Linden~\cite{MC-FR-TVdL}) can be computed for the connected components functor $\pi_{0,\V}\colon \XMod(\V)\to \V$, and more generally, for all functors ${\pi_{0,\XMod^{n}(\V)}^m\colon \XMod^{m+n}(\V)\to \XMod^{n}(\V)}$ where $m>1$.

\section{Cosmash products and internal crossed modules}\label{sec - cosmash products and XMOD}

To keep the proofs in Section~\ref{sec - projective and free} as concise as possible, we avoid the original definition of internal crossed modules from~\cite{Janelidze}. Instead, we adopt the characterization of Hartl and Van der Linden~\cite{HVdL}, which makes crucial use of binary and ternary cosmash products. This point of view naturally aligns with the use of the \emph{action cores} rather than the traditional approach towards internal actions as algebras over a monad. Of course, for all this to make sense, the base category must satisfy certain reasonable assumptions. Therefore, in this paper, we assume that all categories are semi-abelian.

\subsection{Binary cosmash products}

Throughout, we consider a semi-abelian category $\C$.

\begin{definition}\cite{Smash,HVdL,MM-NC}
	Let $A$ and $B$ two objects of $\C$. We can define the comparison morphism
	\[
		\Sigma_{A,B}\coloneq
		\begin{pmatrix}
			1_A & 0 \\ 0 & 1_B
		\end{pmatrix}
		\colon A+B\to A\times B.
	\]
	This canonical arrow is a regular epimorphism, since the category is Bourn~unital~\cite{B0} (by Bourn~protomodularity, all semi-abelian categories are).

	The \defn{cosmash product of $A$ and $B$} is defined by taking the kernel of $\Sigma_{A,B}$:
	\[
		\xymatrix{
		0\ar[r] & A\cosmash B \ar@{{ |>}->}[r]^-{h_{A,B}}& A+B\ar@{-{ >>}}[r]^-{\Sigma_{A,B}} & A\times B\ar[r]&0
		}
	\]
	The construction is functorial: we can define the cosmash product of the arrows as the induced morphism between the kernels on the left-hand side of the diagram
	\[
		\xymatrix{
		0\ar[r] & A\cosmash B\ar@{-->}[d]_-{f\cosmash g} \ar@{{ |>}->}[r]^-{h_{A,B}}& A+B \ar[d]_-{f+g}\ar[r]^-{\Sigma_{A,B}} & A\times B \ar[d]^-{f\times g}\\
		0\ar[r] & A'\cosmash B' \ar@{{ |>}->}[r]_-{h_{A',B'}}& A'+B'\ar[r]_-{\Sigma_{A',B'}} & A'\times B'
		}
	\]
\end{definition}

\subsection{Ternary cosmash products}

\begin{definition}\cite{HVdL}
	Consider three objects $A,B,C$. We can define the comparison morphism
	\[
		\Sigma_{A,B,C}\colon A+B+C\to (A+B)\times (A+C) \times (B+C)
	\]
	where
	\[
		\Sigma_{A,B,C}\coloneq
		\begin{pmatrix}
			\iota_A & \iota_B & 0       \\
			\iota_A & 0       & \iota_C \\
			0       & \iota_B & \iota_C
		\end{pmatrix}
	\]
	and the morphisms $\iota_A$, $\iota_B$ and $\iota_C$ are the canonical inclusions of the corresponding coproducts.
	The \defn{ternary cosmash product} of $A$, $B$ and $C$ is the kernel of $\Sigma_{A,B,C}$ and it is written as
	\[
		\xymatrix{A\cosmash B \cosmash C \ar@{{ |>}->}[rr]^-{h_{A,B,C}} && A+B+C}
	\]
\end{definition}
Notice that $\Sigma_{A,B,C}$ need no longer be a regular epimorphism.
\begin{lemma}\cite{HVdL, MM-NC}\label{lemma - cosmash nice with reg epi}
	The binary and the ternary cosmash products preserve regular epimorphisms.\noproof
\end{lemma}

\begin{remark}[Comparisons between ternary and binary cosmash products~\cite{HVdL}]
	We can define a first ``folding morphism (on the left)''
	\[
		S_{2,1}^{A,B}\colon A\cosmash A \cosmash B \to A\cosmash B
	\]
	via the diagram
	\[
		\xymatrix{
		0 \ar[r]& A\cosmash A \cosmash B \ar@{{ |>}->}[r]^-{h_{A,A,B}} \ar@{-->}[d]_-{S_{2,1}^{A,B}} & A+A+B \ar[r]^-{\Sigma_{A,A,B}} \ar[d]|-{\couniv{1_A}{1_A}+1_B} & (A+A)\times (A+B)^2\ar[d]^-{\couniv{1_A}{1_A}\times \left(\couniv{0}{1_B}\pi_2 \right)}\\
		0 \ar[r]& A\cosmash B \ar@{{ |>}->}[r]_-{h_{A,B}} & A+B \ar[r]_-{\Sigma_{A,B}} & A\times B
		}
	\]
	Similarly, we can define the ``folding morphism (on the right)''
	\[
		S^{A,B}_{1,2}\colon A\cosmash B \cosmash B \to A\cosmash B
	\]
	via the diagram
	\[
		\xymatrix{
		0 \ar[r]& A\cosmash B \cosmash B \ar@{{ |>}->}[r]^-{h_{A,B,B}} \ar@{-->}[d]_-{S_{1,2}^{A,B}} & A+B+B \ar[r]^-{\Sigma_{A,B,B}} \ar[d]|-{1_A + \couniv{1_B}{1_B}} & (A+B)^2\times (B+B)\ar[d]^-{\left(\couniv{1_A}{0}\pi_1 \right)\times \couniv{1_B}{1_B}}\\
		0 \ar[r]& A\cosmash B \ar@{{ |>}->}[r]_-{h_{A,B}} & A+B \ar[r]_-{\Sigma_{A,B}} & A\times B
		}
	\]
\end{remark}

\begin{proposition}\cite{Actions}
	The ternary cosmash $A\cosmash B \cosmash C$ is the subobject of $(A+B)\cosmash C$ represented by the morphism $j_{A,B,C}$ defined in the diagram
	\[
		\xymatrix{
		A\cosmash B \cosmash C \ar@{{ |>}->}[dr]^-{h_{A,B,C}}\ar@{-->}[d]_-{j_{A,B,C}}  & & \\
		(A+B)\cosmash C \ar@{{ |>}->}[r]_-{h_{(A+B),C}} & A+B+C \ar[r]^-{\Sigma_{(A+B),C}} \ar[rd]_-{\Sigma_{A,B,C}} & (A+B)\times C \\
		& & (A+B)\times (A+C)\times (B+C)\ar[u]_-{1_{A+B}\times \couniv{0}{1_C}\pi_{2}}
		}
	\]
	In particular, replacing $C$ by $B$, the equality
	\[
		S_{2,1}^{A,B}=\left(\couniv{1_A}{1_A}\cosmash 1_B\right) j_{A,A,B}.
	\]
	is obtained.\noproof
\end{proposition}

\begin{lemma}\label{lemma - useful lemma}
	Let $f\colon A\to A'$, $g\colon B\to A'$ and $l\colon B\to B'$ be three morphisms. Then the equation
	\[
		\left(\couniv{f}{g}\cosmash l\right) j_{A,B,B} = S_{2,1}^{A',B'} (f\cosmash g \cosmash l)
	\]
	holds.
\end{lemma}
\begin{proof}
	The diagram
	\[
		\resizebox{\textwidth}{!}{\xymatrix{
		A\cosmash B \cosmash B \ar[rr]_-{j_{A,B,B}} \ar@/^2pc/[rrr]^-{h_{A,B,B}} \ar[d]^-{f\cosmash g \cosmash l}\ar@/_2.5pc/[ddd]_-{h_{A,B,B}} && (A+B)\cosmash B\ar[r]_-{h_{A+B,B}}\ar[d]|-{\couniv{f}{g}\cosmash l} & (A+B)+B \ar[d]^-{\couniv{f}{g}+l}\\
		A'\cosmash A'\cosmash B' \ar[rr]^-{S_{2,1}^{A',B'}}\ar[dr]_-{j_{A',A',B'}}&& A'\cosmash B' \ar[r]^-{h_{A',B'}} & A'+B'\\
		& (A'+A')\cosmash B'\ar[ru]|-{\couniv{1_{A'}}{1_{A'}}\cosmash 1_{B'}} \ar[rr]_-{h_{(A'+A'),B'}} && (A'+A')\cosmash B'\ar[u]_-{\couniv{1_{A'}}{1_{A'}}+1_{B'}}\\
		A+B+B \ar@/_1pc/[rrru]_-{f+g+l} &&&
		}}
	\]
	commutes, and $h_{A',B'}$ is a monomorphism.
\end{proof}

\begin{lemma}\label{lemma - second useful lemma}
	Let $f\colon A\to A'$ and $g\colon B\to B'$ be two morphisms. Then the equality
	\[
		(f\cosmash g) S_{1,2}^{A,B}= S^{A',B'}_{1,2} (f\cosmash g \cosmash g)
	\]
	holds.
\end{lemma}
\begin{proof}
	The equality holds since the outer diagram
	\[
		\xymatrix@C=5pc{
		& A+B+B\ar[dr]^-{1_A+\couniv{1_B}{1_B}}& \\
		A\cosmash B\cosmash B \ar[ru]^-{h_{A,B,B}}\ar[r]_-{S_{1,2}^{A,B}}\ar[d]^-{f\cosmash g\cosmash g} \ar@/_2.5pc/[dd]_-{h_{A,B,B}} & A\cosmash B\ar[d]|-{f\cosmash g} \ar@{{ |>}->}[r]_-{h_{A,B}} & A+B\ar[d]^-{f+g} \\
		A' \cosmash B' \cosmash B'\ar[r]^-{S_{1,2}^{A',B'}} \ar[dr]|-{h_{A',B',B'}} & A'\cosmash B'\ar@{{ |>}->}[r]^-{h_{A',B'}} & A'+B'\\
		A+B+B \ar[r]_-{f+g+g}	& A'+B'+B'\ar[ru]|-{1_{A'}+\couniv{1_{B'}}{1_{B'}}} &
		}
	\]
	commutes, and $h_{A',B'}$ is a monomorphism.
\end{proof}
The next lemma shows how the ternary cosmash product appears naturally when taking a cosmash product of a coproduct.
\begin{lemma}\label{lemma - covering of A+B flat C}\cite{HVdL,Actions}
	The induced morphism
	\[
		\counivtr{j_{A,B,C}}{\iota_1\cosmash 1_C}{\iota_2 \cosmash 1_C}\colon (A\cosmash B \cosmash C)+ (A\cosmash C) +(B\cosmash C) \to (A+B)\cosmash C
	\]
	where $\iota_1$ and $\iota_2$ are the canonical inclusions of the coproduct $A+B$, is a regular epimorphism.\noproof
\end{lemma}

\subsection{Actions and cosmash products}\label{subsec - actions and semidirect product}

A classical way of defining an action of an object $A$ on an object $X$ is to consider a split extension where $X$ is the kernel and $A$ is the cokernel. This information can be encoded in a split short exact sequence such as~\eqref{eq - action via split extension}.
\begin{equation}
	\xymatrix{0 \ar[r] & X \ar@{{ |>}->}[r]^-{k} & E \ar@{-{ >>}}@<.55ex>[r]^-{p} & A \ar[r] \ar@{{ >}->}@<.55ex>[l]^-{s} & 0}.\label{eq - action via split extension}
\end{equation}
Since the category is semi-abelian, we can view this as an algebra over a certain monad $T_A$ \cite{BJK,Bourn-Janelidze:Semidirect}. The object $T_A(X)$ is the kernel $\kappa_{A,X}\colon A\flat X \to A+X$ of the induced morphism $\couniv{1_A}{0}\colon {A+X\to A}$. The components $\eta_{A,X}\colon X\to A\flat X$ of the unit are the morphisms induced by the second inclusion of the binary coproduct: let $X\in \C$, we have $\iota_2 =\kappa_{A,X}\eta_{A,X}\label{eq-unit adjunction monad action}$ (e.g.\ \cite{BJK}).

However, this algebra can be replaced by its so-called \emph{core}~\cite{HLoiseauarXiv,Actions}. Indeed, an algebra $\xi\colon A\flat X\to X$ of the monad $T_A=A\flat-$ satisfies the unit axiom, i.e.\ the right-hand triangle
\[
	\xymatrix@C=3pc @R=2pc{
	0\ar[r] & A\cosmash X\ar@{-->}[dr]_-{\psi} \ar@{{ |>}->}[r]^-{i_{A,X}} & A\flat X \ar@{-{ >>}}@<.55ex>[r]^-{\tau_{A,X}}\ar[d]^-{\xi} & X \ar[r] \ar@{{ >}->}@<.55ex>[l]^-{\eta_{A,X}}\ar[dl]^-{1_X} & 0\\
	&& X &&
	}
\]
commutes, where ${\tau_{A,X}\coloneq \couniv{0}{1_X}\kappa_{A,X}}$ (see for instance~\cite{dMVdL19} for details). Hence the pair of monomorphisms $(i_{A,X},\eta_{A,X})$ is (strongly) epimorphic and therefore the morphism $\xi$ is uniquely determined by its restriction $\psi$ (along $i_{A,X}$) to $A\cosmash X$.

Hence $\psi$ may be viewed as a fragment of $\xi$ containing all its information. This point of view of considering the \defn{action core} $\psi$ of $\xi$ instead of the  ``classical'' action $\xi$ is compatible with an equivalent characterization of internal crossed modules (obtained by Hartl and Van der Linden~\cite{HVdL}) (see Section~\ref{sec - Xmod and cosmash}). This point of view will shorten the proofs in this article.

\begin{remark}[Construction of semi-direct products via action cores]
	In the category of groups, given an action, we can construct a split extension where the middle object is the so-called \defn{semi-direct product of groups} (see for instance \cite{Bourn-Janelidze:Semidirect}).

	As explained in \cite{Actions,HVdL}, with a split extension as in~\eqref{eq - action via split extension} we can associate a unique action core\footnote{Constructing a split extension out of a morphism $A\cosmash X\to X$ is possible under certain conditions (see~\cite{HVdL} for further details).}: given a split extension~\eqref{eq - action via split extension}, we can define a unique ${\psi\colon A\cosmash X \to X}$ as the restriction to the kernels of the induced morphism $\couniv{s}{k}$
	\begin{equation}
		\vcenter{\xymatrix{
		0\ar[r] & A\cosmash X\ar@{-->}[d]_-{\psi} \ar@{{ |>}->}[r]^-{h_{A,X}} & A+X \ar@{->>}[d]^-{\couniv{s}{k}} \ar@{-{ >>}}[r]^-{\Sigma_{A,X}} & A\times X\ar[d]^-{\pi_1}\ar[r] & 0 \\
		0 \ar[r] & X \ar[ru]|-{\iota_2} \ar@{{ |>}->}[r]^-{k} & E \ar@{-{ >>}}@<.5ex>[r]^-{p} & A \ar[r] \ar@{{ >}->}@<.5ex>[l]^-{s} & 0
		}}\label{diag - uniqueness of action for split extension}
	\end{equation}
	By protomodularity, it follows that $\couniv{s}{k}$ is a regular epimorphism. Moreover, we can prove that it is the coequalizer of the pair ($h_{A,X}, \iota_2 \psi$). We usually call the object $E$ the \defn{semi-direct product of $X$ and $A$ along $\psi$} and write it as $X\rtimes_\psi A$.
\end{remark}

\subsection{Internal crossed modules via action cores}\label{sec - Xmod and cosmash}

The first internal definition of crossed modules for a given semi-abelian is given by G.~Janelidze~\cite{Janelidze}. The idea of the definition is to conserve the well-known equivalence of categories between crossed modules and internal categories in groups~\cite{BrownSpencer}. Therefore an internal crossed module (in the sense of G.~Janelidze) is equivalent to an internal category in the given semi-abelian category.

However, this approach involves the monad $A\flat -$ in the construction, and in this article, we will prefer the action core to simplify the proofs. That is the reason why we will consider another, equivalent characterization of internal crossed modules:
\begin{definition}\cite[Theorem 5.6.]{HVdL}\label{thm - XMod and cosmash}
	Let $X$ and $A$ two objects in a semi-abelian category $\C$, consider a morphism $\partial\colon X\to A$ (called the \defn{boundary morphism}) in $\C$ and an action core $\psi\colon A\cosmash X\to X$ associated to a split extension in $\C$ as in~\eqref{eq - action via split extension}. The $4$-tuple $(X,A,\psi,\partial)$ is an \defn{internal crossed module} if the following conditions hold:
	\begin{enumerate}
		\item the \defn{precrossed module condition}, the diagram
		      \[
			      \xymatrix{
			      A\cosmash X\ar[r]^-{1_A\cosmash \partial} \ar[d]_-{\psi} & A\cosmash A\ar[d]^-{\overline{\chi_A}}\\
			      X \ar[r]_-{\partial}& A
			      }
		      \]
		      where $\overline{\chi_A}\coloneq \couniv{1_A}{1_A} h_{A,A}$ is the \defn{conjugation action core}, commutes;
		\item the \defn{Peiffer condition}, the diagram
		      \[
			      \xymatrix{
			      X\cosmash X\ar[r]^-{\partial \cosmash 1_X} \ar[d]_-{\overline{\chi_X}}& A\cosmash X \ar[d]^-{\psi} \\
			      X\ar[r]_-{1_X} & X
			      }
		      \]
		      commutes;
		\item the \defn{ternary commutator condition}, the diagram
		      \[
			      \xymatrix{A\cosmash X \cosmash X\ar[r]^-{\psi_{1,2}^{A,X}}\ar[d]_-{1_A\cosmash \partial \cosmash 1_X} & X\ar[d]^-{1_X} \\
			      A \cosmash A \cosmash X\ar[r]_-{\psi_{2,1}^{A,X}} & X}
		      \]
		      where $\psi_{1,2}^{A,X}\coloneq \psi S_{1,2}^{A,X}$ and $\psi_{2,1}^{A,X}\coloneq \psi S_{2,1}^{A,X}$, commutes.
	\end{enumerate}
	A \defn{crossed module morphism} $(f_X,f_A)\colon(X,A,\psi,\partial)\to (X',A',\psi',\partial')$ is a pair of morphisms $f_X\colon X\to X'$, $f_A\colon A\to A'$ in $\C$ such that $f_A \partial = \partial' f_X$ and the diagram
	\begin{equation}
		\vcenter{\xymatrix{
		A\cosmash X\ar[r]^-{\xi}\ar[d]_-{f_A\cosmash f_X} & X\ar[d]^-{f_X}\\
		A'\cosmash X'\ar[r]_-{\xi'} & X'
		}}\label{diag - morphism XMOD actions}
	\end{equation}
	commutes. The diagram~\eqref{diag - morphism XMOD actions} expresses the \defn{equivariance} of the pair $(f_X,f_A)$ with respect to the action cores $\psi$ and $\psi'$.

	The above data describe a category, denoted by $\XMod(\C)$, called the \defn{category of internal crossed modules in $\C$}.
\end{definition}

\begin{remark}
	The condition expressed by~\eqref{diag - morphism XMOD actions} is equivalent to the existence of the following morphism of split short exact sequences:
	\begin{equation}
		\vcenter{\xymatrix{
		0 \ar[r] & X \ar@{{ |>}->}[r]^-{k} \ar[d]_-{f_X}& X\rtimes_\psi A \ar@{-{ >>}}@<.5ex>[r]^-{p} \ar[d]|-{f_X\rtimes f_A}& A \ar[d]_-{f_A} \ar[r] \ar@{{ >}->}@<.5ex>[l]^-{s} & 0\\
		0 \ar[r] & X' \ar@{{ |>}->}[r]^-{k'} & X'\rtimes_{\psi'} A' \ar@{-{ >>}}@<.5ex>[r]^-{p'} & A' \ar[r] \ar@{{ >}->}@<.5ex>[l]^-{s'} & 0
		}}\label{diag - morphism XMOD SSES}
	\end{equation}
	Here the action cores $\psi$ and $\psi'$ correspond, respectively, to the top and the bottom split extension of $A$ and $A'$.
\end{remark}

\begin{example}[Kernels are internal crossed modules]\label{ex - kernel and XMod}
	A classic example of a crossed module in $\Gp$ is given by a group $G$ and a normal subgroup\footnote{Let's recall that a subgroup is normal if and only if the canonical inclusion is a kernel of a group morphism.} $H$ of $G$. This forms a crossed module with the conjugation action core and the inclusion of $H$ into $G$ as the boundary morphism.

	In a semi-abelian category, if we consider a kernel $k\colon Y\to X$, then it is also a clot~\cite{Agliano-Ursini,Janelidze-Marki-Ursini} and vice versa~\cite{MM-NC}. Hence, there exists a (unique) action $\chi\colon X\flat Y \to Y$ such that the diagram
	\begin{equation}\label{diag - clot def}
		\vcenter{\xymatrix{
		X\flat Y \ar[r]^-{\chi} \ar[d]_-{1_X\flat k} & Y\ar@{{ |>}->}[d]^-{k}\\
		X\flat X \ar[r]_-{\chi_X} & X
		}}
	\end{equation}
	where $\chi_X\coloneq \couniv{1_X}{1_X} \kappa_{X,X}$ (the \emph{conjugation action}), commutes.

	Here, we consider the action cores associated with the actions $\chi$ and $\chi_X$. As a consequence we have
	\[
		\xymatrix{
		X\cosmash Y \ar[r]^-{i_{X,Y}}\ar[d]_-{1_X\cosmash k}& X\flat Y \ar[r]^-{\chi} \ar[d]_-{1_X\flat k} & Y\ar@{{ |>}->}[d]^-{k}\\
		X\cosmash X\ar[r]_-{i_{X,X}} &X\flat X \ar[r]_-{\chi_X} & X
		}
	\]
	We can prove that the quadruple $(X,Y,\chi i_{X,Y},k)$ is a crossed module. For the sake of simplicity, we introduce the notation $\overline{\chi}\coloneq \chi i_{X,Y}$ and $\overline{\chi_X}\coloneq \chi_X i_{X,X}$. In fact, $\overline{\chi}$ is the only action endowing $(X,Y,k)$ with a crossed module structure (see for instance~\cite{DevalVdL}).

	Concerning morphisms between such examples of internal crossed modules, if we consider a pair of morphisms $(f_Y,f_X)$ such that the square
	\[
		\xymatrix{
		Y\ar@{{ |>}->}[r]^-{k}\ar[d]_-{f_Y} & X \ar[d]^-{f_X} \\
		Y'\ar@{{ |>}->}[r]_-{k'} & X'
		}
	\]
	commutes, then this defines a morphism of crossed modules. Indeed, the equivariance of the pair ($f_Y,f_X$) is automatic since we have
	\begin{align*}
		k'f_Y\overline{\chi} & = f_X k \overline{\chi} = f_X \overline{\chi_X}(1_Y\cosmash k) = \overline{\chi_{X'}}(f_X\cosmash f_X)(1_Y\cosmash k) = \overline{\chi_{X'}}(f_X\cosmash f_X k) \\
		                     & = \overline{\chi_{X'}}(f_X\cosmash k'f_Y) = \overline{\chi_{X'}}(1_{X'}\cosmash k')(f_X\cosmash f_Y) = k \overline{\chi'}(f_X\cosmash f_X)
	\end{align*}
\end{example}

\begin{lemma}\label{lemma - equivalence clot cosmash}
	Consider a kernel $k\colon Y\to X$ and the associated action cores $\overline{\chi}$ and $\overline{\chi_X}$ (see Example~\ref{ex - kernel and XMod}). Then $k\overline{\chi}=\couniv{1_X}{k}h_{X,Y}$.
\end{lemma}
\begin{proof}
	It suffices to notice that
	\begin{align*}
		\chi_X (1_X\flat k)= \couniv{1_X}{1_X} \kappa_{X,X}(1_X\flat k)
		=\couniv{1_X}{1_X}(1_X+k)\kappa_{X,Y}=\couniv{1_X}{k}\kappa_{X,Y}
	\end{align*}
	and to compose with $i_{X,Y}$.
\end{proof}

\section{Projective and free objects}\label{sec - projective and free}

In the category of groups, it is well known that free objects are stable under subobjects. This makes it an example of a so-called \defn{Schreier variety of algebras}, i.e., a variety where the free objects are closed under subobjects.

One of the main consequences of being a Schreier variety is that the classes of projective and free objects coincide. Hence, this is necessary for being Schreier.

This necessary condition has been used in~\cite{Carrasco-Homology} to prove that $\XMod(\Gp)$ is not a Schreier variety. We claim the same holds if we replace $\Gp$ by any semi-abelian variety~$\V$. To do so, in Section~\ref{sec - projective object}, we will construct an internal crossed module that is projective, and then compare it to the free objects in $\XMod(\V)$ (cf. Section~\ref{sec - free object}).

\subsection{Projective objects in $\SSE_B(\C)$}\label{subsec-SSEB(V)}

We write $\SSE_B(\C)$ for the category of split extensions over a fixed object $B$ in a semi-abelian category $\C$.

\begin{lemma}\label{lemma - reg epi in H-group}
	Consider an object $B$ in a semi-abelian category $\C$. A morphism $(f,g)$ between two split extensions $(X,E)$ and $(X',E')$ over $B$ is a regular epimorphism in the category $\SSE_B(\C)$ if and only if the morphism $f$ is a regular epimorphism in $\C$.
\end{lemma}
\begin{proof}
	Consider a morphism
	\[
		\xymatrix{
		0 \ar[r] & X \ar[d]^-{f} \ar@{{ |>}->}[r] & E \ar[d]^-{g} \ar@{-{ >>}}@<.55ex>[r] & B \ar@{=}[d] \ar[r] \ar@{{ >}->}@<.55ex>[l] & 0\\
		0 \ar[r] & X' \ar@{{ |>}->}[r] & E' \ar@{-{ >>}}@<.55ex>[r] & B \ar[r] \ar@{{ >}->}@<.55ex>[l] & 0
		}
	\]
	in $\SSE_B(\C)$.

	It is well known that the above morphism of split extensions is a regular epimorphism if and only if $f$ and $g$ are regular epimorphisms in $\C$. In addition, $g$ is a regular epimorphism if and only if $f$ is a regular epimorphism~\cite[Lemma 4.2.5]{Borceux-Bourn} since $\C$ is a semi-abelian category.
\end{proof}

A projective object in the category $\SSE_B(\C)$ is a split extension of $B$ such that any regular epimorphism of split extensions of $B$

\begin{equation}
	\vcenter{\xymatrix{
	0 \ar[r] & A\ar@{->>}[d]^-{f} \ar@{{ |>}->}[r] & C\ar@{->>}[d]^-{g} \ar@{-{ >>}}@<.5ex>[r] & B \ar[r] \ar@{{ >}->}@<.5ex>[l] & 0\\
	0 \ar[r] & Q \ar@/^1pc/@{-->}[u]^-{f'} \ar@{{ |>}->}[r]& Z \ar@/^1pc/@{-->}[u]^-{g'} \ar@{-{ >>}}@<.5ex>[r] & B\ar@{=}[u] \ar[r] \ar@{{ >}->}@<.5ex>[l] & 0
	}}\label{eq - def projective split extension of B}
\end{equation}
is a split epimorphism in $\SSE_B(\C)$ with the morphism of split extensions $(f',g')$ as section. It is important to notice that the identity morphism on the bottom split short exact sequence in~\eqref{eq - def projective split extension of B} is given by the pair $(1_Q,1_Z)$. This implies that $f$ and~$g$ are split epimorphisms in $\C$ with, respectively, $f'$ and $g'$ as sections.

\begin{lemma}\label{lemma - example projective in SSEP}
	For a given object $B$ in a semi-abelian category $\C$, if $X$ is a projective object of $\C$ then
	\[
		\xymatrix{
		0\ar[r] & B\flat X \ar@{{ |>}->}[r]^-{\kappa_{B,X}} & B+X \ar@{-{ >>}}@<.5ex>[r]^-{\couniv{1_B}{0}} & B\ar[r] \ar@{{ >}->}@<.5ex>[l]^-{\iota_1}  & 0
		}
	\]
	is a projective object in $\SSE_B(\C)$.
\end{lemma}

\begin{proof}
	We consider a morphism of split short exact sequences where both $f_L$ and $f_Y$ are regular epimorphisms in $\C$.
	\[
		\xymatrix{
		0 \ar[r] &L  \ar@{{ |>}->}[r]_-k \ar@{->>}[d]^-{f_L} & Y \ar@{-{ >>}}@<.5ex>[r]^-g \ar@{->>}[d]^-{f_Y} & B \ar@{{ >}->}@<.5ex>[l]^-t \ar[r] \ar@{=}[d]& 0\\
		0\ar[r] & B\flat X \ar@/^1pc/@{-->}[u]^-{\overline{\couniv{t}{h}}} \ar@{{ |>}->}[r]^-{\kappa_{B,X}} & B+X \ar@/^1pc/@{-->}[u]^-{\couniv{t}{h}} \ar@{-{ >>}}@<.5ex>[r]^-{\couniv{1_B}{0}} & B \ar[r]\ar@{{ >}->}@<.5ex>[l]^-{\iota_1} & 0
		}
	\]

	Since $X$ is a projective object in $\C$, there exists $h\colon X\to Y$ such that $\iota_2=f_Y h$. Hence the induced morphism $\couniv{t}{h}\colon B+X\to Y$ is a section of $f_Y$:
	\[f_Y \couniv{t}{h}\iota_2 = f_Y h =\iota_2 \text{ and } f_Y \couniv{t}{h}\iota_1 = f_Y t =\iota_1.\]
	Moreover, we have $\couniv{t}{h}\iota_1=t$ and also $g\couniv{t}{h}=\couniv{1_B}{0}$ since
	\[g\couniv{t}{h}\iota_2=gh=\couniv{1_B}{0}f_Y h=\couniv{1_B}{0}\iota_2 \text{ and } g\couniv{t}{h}\iota_1=gt=1_B=\couniv{1_B}{0}\iota_1.\]
	This implies in particular the existence of the morphism $\overline{\couniv{t}{h}}$ which is the restriction of $\couniv{t}{h}$ to the kernels.

	It is obvious that the pair $(\overline{\couniv{t}{h}},\couniv{t}{h})$ defines a morphism in $\SSE_B(\V)$. Moreover, ${\overline{\couniv{t}{h}}}$ is a section of $f_L$.
\end{proof}
\begin{remark}
	The previous result can also be reformulated as follows: for any object $B$ in a semi-abelian category $\C$, the free $B$-action on a projective object is itself projective in $\SSE_B(\C)$. Here, the freeness refers to the adjunction associated with the monad $T_B$ (e.g.\ \cite{BJK}). An alternative proof is to show that the right adjoint $\ker\colon \Pt_B(\C)\to \C$ (where $\Pt_B(\C)$ denotes the category of points over $B$) preserves regular epimorphisms. As a result, the left adjoint $B+(-)\colon \C\to \Pt_B(\C)$ preserves projective objects (see for instance~\cite{Hilton-Stammbach}).
\end{remark}

\begin{lemma}\label{lemma - H-group enough projectives}
	For a given object $B$ in a semi-abelian category $\C$ with enough projectives, the category $\SSE_B(\C)$ has enough projectives.
\end{lemma}
\begin{proof}
	Let's consider the split extension
	\begin{equation}
		\xymatrix{
		0\ar[r] & X \ar@{{ |>}->}[r]^-k & X\rtimes_\psi B \ar@{-{ >>}}@<.5ex>[r]^-f & B\ar[r] \ar@{{ >}->}@<.5ex>[l]^-s & 0}.\label{diag - starting SSES A}
	\end{equation}

	Since the category $\C$ has enough projectives, the object $X$ is a regular quotient via $p$ of a projective object $R_X$. We have
	\begin{equation}
		\begin{aligned}
			\xymatrix{
			0\ar[r]  & B\flat R_X \ar@{-->>}@/_2pc/[dd]_-{\overline{\couniv{s}{kp}}}\ar@{{ |>}->}[dr]^-{\kappa_{B,R_X}} &                                                                         &                                                         \\
			         & R_X\ar[u]_-(0.3){\eta_{R_X}}\ar[r]^-{\iota_2}\ar@{->>}[d]^-{p}                                   & B+R_X\ar[d]^-{\couniv{s}{kp}} \ar@{-{ >>}}@<.5ex>[r]^-{\couniv{1_B}{0}} & B\ar@{=}[d]\ar@{{ >}->}@<.5ex>[l]^-{\iota_1} \ar[r] & 0 \\
			0 \ar[r] & X \ar@{{ |>}->}[r]_-k                                                                            & X\rtimes_\phi B \ar@{-{ >>}}@<.5ex>[r]^-f                               & B \ar@{{ >}->}@<.5ex>[l]^-s \ar[r]                  & 0
			}\label{diag-SSES A regular quotient}
		\end{aligned}
	\end{equation}
	where $\overline{\couniv{s}{kp}}$ is the restriction of $\couniv{s}{kp}$ to the kernels and $\eta_{R_X}$ is the unit component of the monad $T_B$ (see p~\pageref{eq-unit adjunction monad action}). Since $\overline{\couniv{s}{kp}}$ is a regular epimorphism, by Lemma~\ref{lemma - reg epi in H-group}, we have a regular epimorphism in $\SSE_B(\C)$.

	This proves that the bottom split extension of $B$ in~\eqref{diag-SSES A regular quotient} is a regular quotient of the top split extension. Finally, this top split extension of $B$ is a projective object in $\SSE_B(\C)$ by Lemma~\ref{lemma - example projective in SSEP}.
\end{proof}

An important class of semi-abelian categories having enough projectives is the class of semi-abelian varieties of algebras. Let $\V$ be such a variety. Then any object $X\in \V$ is the regular quotient of the free object $F_r(U(X))$ where $U(X)$ is the underlying set. As a consequence of the proof above, the split extension
\begin{equation}
	\xymatrix@C=4pc{
	0\ar[r] & B\flat F_r(U(X)) \ar@{{ |>}->}[r]^-{\kappa_{B,F_r(U(X))}} & B+F_r(U(X)) \ar@{-{ >>}}@<.5ex>[r]^-{\couniv{1_B}{0}} & B \ar[r]\ar@{{ >}->}@<.5ex>[l]^-{\iota_1} & 0
	}
	\label{eq - prototype projective in SSEP(C)}
\end{equation}
of $B$ is a projective regular cover of \eqref{diag - starting SSES A} in $\SSE_B(\V)$.

\begin{corollary}
	For a given object $B$ in a semi-abelian category $\C$ with enough projectives, any projective object in $\SSE_B(\C)$ as~\eqref{diag - starting SSES A} is a retract of~\eqref{eq - prototype projective in SSEP(C)}.\noproof
\end{corollary}

\subsection{A projective object in internal crossed modules}\label{sec - projective object}
Carrasco, Cegarra, and Grandje\'an construct a projective object in $\XMod(\Gp)$ which is not a free crossed module.
\begin{proposition}\cite{Carrasco-Homology}
	In $\XMod(\Gp)$, if $P$ is a projective group and $Q$ is a projective $P$-group then the inclusion morphism $Q\to Q\rtimes P$ is a projective crossed module.\noproof
\end{proposition}

This construction can also be done in $\XMod(\V)$, for any semi-abelian category~$\V$.

\begin{theorem}\label{thm - projective in XMod(C)}
	If $P$ is a projective object in $\C$ and if the split extension
	\begin{equation}
		\xymatrix{
		0 \ar[r] & Q \ar@{{ |>}->}[r]^-{\partial'} & Z \ar@{-{ >>}}@<.5ex>[r]^-{p} & P \ar[r] \ar@{{ >}->}@<.5ex>[l]^-{s} & 0
		} \label{eq - split ext of P}
	\end{equation}
	is a projective object in the category of split extensions of $P$ (see for instance~\eqref{eq - def projective split extension of B}), then the kernel $\partial'$, endowed with the natural internal crossed module structure of Example~\ref{ex - kernel and XMod}, is a projective object in $\XMod(\C)$.
\end{theorem}

Before doing the proof, we need to work on some preliminary results.
\begin{remark}
	As explained in Section~\ref{subsec - actions and semidirect product}, without loss of generality we can replace the object $Z$ by $Q\rtimes_{\psi} P$, where $\psi$ is the unique action core associated to the extension~\eqref{eq - split ext of P}.

	By Example~\ref{ex - kernel and XMod}, we can see the triple $(Q, Q\rtimes_\psi P,\partial')$ as an internal crossed module.
\end{remark}
Consider a regular epimorphism $(f_T,f_G)\colon (T,G, \phi, \partial)\to (Q,Q\rtimes_\psi P, \overline{\chi}, \partial')$ in $\XMod(\C)$:
\[
	\xymatrix{
	T\ar[r]^-{\partial} \ar@{->>}[d]_-{f_T} & G\ar@{->>}[d]^-{f_G} & \\
	Q\ar@{{ |>}->}[r]_-{\partial'} & Q \rtimes_\psi P \ar@<.5ex>[r]^-{p} & P \ar@<.5ex>[l]^-{s}
	}
\]
that is, a morphism where $f_G$ and $f_T$ regular epimorphisms in $\C$ (see e.g.\ \cite{dMVdL19.4}). To prove that $(Q,Q\rtimes_\psi P, \overline{\chi}, \partial')$ is a projective object, it suffices to find a section in $\XMod(\C)$ for the pair of morphisms $(f_T,f_G)$.

The proof will be divided into several consecutive lemmas.

In the following lemmas, we will use the same notations and assumptions as in the statement of Theorem~\ref{thm - projective in XMod(C)}.

\begin{lemma}[Lifting over $f_G$]
	There exists a lifting $g_1$ of $s$ over $f_G$.
\end{lemma}
\begin{proof}
	Since $P$ is projective in $\C$ and since $f_G$ is a regular epimorphism in $\C$,
	\[
		\xymatrix{
		T\ar[r]^-{\partial} \ar@{->>}[d]_-{f_T} & G\ar@{->>}[d]^-{f_G} & \\
		Q\ar@{{ |>}->}[r]_-{\partial'} & Q \rtimes_\psi P \ar@<.55ex>[r]^-{p} & P \ar@<.55ex>[l]^-{s}\ar@/_1pc/@{-->}[lu]_-{g_1}
		}
	\]
	there exists a morphism $g_1\colon P \to G$ such that $s=f_G g_1$.
	\qedhere
\end{proof}

\begin{lemma}[A section of $f_T$]\label{lemma - def hT}
	We can find a section $g_T$ of the regular epimorphism $f_T$ such that $g_T$ is equivariant with respect to the action cores $\psi$ and $\phi (g_1\cosmash 1_T)$.
\end{lemma}
\begin{proof}
	We can define an action core of $P$ to $T$ via the morphism
	\[ \phi'=\phi (g_1\cosmash 1_T)\colon P\cosmash T \to T \]
	Such morphism allows us to define the semi-direct product and therefore a split extension of $P$ with $T$ as the kernel.
	Via the morphism $g_1$, we can see $T$ as the kernel of a split extension of $P$ (e.g.\ \cite[Proposition 3.8]{HVdL}). Finally, we have the morphism of split short exact sequences
	\[
		\xymatrix{
		0 \ar[r] & T \ar@{=}[d] \ar@{{ |>}->}[r]^-{l} & T\rtimes_{\phi'} P \ar[d]|-{1_T\rtimes g_1}\ar@{-{ >>}}@<.55ex>[r]^-{g} & P \ar[d]^-{g_1} \ar[r] \ar@{{ >}->}@<.55ex>[l]^-{t} & 0\\
		0 \ar[r] & T \ar@{{ |>}->}[r]^-{k'} & T\rtimes_\phi G \ar@{-{ >>}}@<.55ex>[r]^-{p'} & G \ar[r] \ar@{{ >}->}@<.55ex>[l]^-{s'} & 0
		}
	\]

	Since $Q$ is a projective object in $\SSE_P(\C)$, there exists a morphism $g_T$ in $\C$ such that
	\[
		\xymatrix{
		0\ar[r]&T\ar@{{ |>}->}[r]^-{l}\ar@{->>}[d]^-{f_T} & T \rtimes_{\phi(g_1\cosmash 1_T)} P \ar@<.55ex>[r]^-{g}\ar@{->>}[d]^-{f_T\rtimes 1_P} & P \ar@<.55ex>[l]^-{t} \ar@{=}[d]\ar[r]& 0\\
		0\ar[r]&Q\ar@{{ |>}->}[r]_-{\partial'} \ar@/^1pc/@{-->}[u]^-{g_T} & Q \rtimes_\psi P \ar@<.55ex>[r]^-{p} \ar@/^1pc/@{-->}[u]^-{g_T\rtimes 1_P} & P \ar@<.55ex>[l]^-{s} \ar[r]& 0
		}
	\]
	the pair $(g_T, g_T\rtimes 1_P)$ is a section in $\SSE_P(\C)$ for $(f_T,f_T\rtimes 1_P)$. Moreover, this implies that $g_T$ is a section in $\C$ of $f_T$ (see page~\pageref{eq - def projective split extension of B}) and  $g_T\psi = \phi (g_1\cosmash 1_T)(1_P\cosmash g_T)$.
\end{proof}

Before proceeding with the next lemma, we notice that ${\psi=\overline{\chi} (s\cosmash 1_Q)}$. Indeed, by the uniqueness of the action core defined by a split extension~\eqref{diag - uniqueness of action for split extension}, it suffices to prove that $\partial'\overline{\chi} (s\cosmash 1_Q) =\couniv{s}{\partial'} h_{P,Q}$. This is true by the commutativity of the diagram
\[
	\xymatrix{
	P \cosmash Q\ar[r]^-{s\cosmash 1_Q} \ar[d]_-{h_{P,Q}} & (Q\rtimes_\psi P)\cosmash Q \ar[r]^-{\chi} \ar[d]_-{h_{Q\rtimes_\psi P, Q}} & Q  \ar[d]^-{\partial'}\\
	P+Q \ar[r]^-{s+1_Q} \ar@/_1pc/[rr]_-{\couniv{s}{\partial'}} & (Q\rtimes_\psi P)+Q \ar[r]^-{\couniv{1_{Q\rtimes_\psi P}}{\partial'}} & Q\rtimes_\psi P
	}
\]
where the right-hand square commutes by Lemma~\ref{lemma - equivalence clot cosmash} and the left-hand square commutes by definition of the binary cosmash.

\begin{lemma}[A section of $f_G$]\label{lemma - def hG}
	There exists a morphism $g_G\colon Q\rtimes_\psi P \to G$ which is a section of $f_G$ and such that $\partial g_T=g_G\partial'$.
\end{lemma}

\begin{proof}
	To define the morphism $g_G$, we will use the construction of $Q\rtimes_\psi P$---its being the coequalizer of the pair $(h_{P,Q}, \iota_2 \psi)$:
	\[
		\xymatrix{
		P\cosmash Q \ar[r]^-{h_{P,Q}} \ar[d]_-{\psi} & P+Q \ar[r]^-{\couniv{g_1}{\partial g_T}} \ar@{->>}[d]|{\couniv{s}{\partial'}} & G\\
		Q\ar[ru]^-{\iota_2}\ar@{{ |>}->}[r]_-{\partial'} & Q \rtimes_\psi P \ar@{-->}[ru]_-(0.6){g_G} \ar@<.55ex>[r]^-{p} & P \ar@<.55ex>[l]^-{s}
		}
	\]
	We can observe that $\couniv{g_1}{\partial g_T}$ coequalizes the pair $(h_{P,Q}, \iota_2 \psi)$. Indeed,
	\begin{align*}
		\couniv{g_1}{\partial g_T}\iota_2 \psi & = \partial g_T \psi
		\overset{(\alpha)}{=} \partial \phi (g_1\cosmash 1_T)(1_P \cosmash g_T)
		\overset{(\beta)}{=} \overline{\chi_G} (1_T\cosmash \partial)(g_1\cosmash g_T) \\
		                                       & = \couniv{1_G}{1_G} (g_1+g_T)h_{P,Q}
		= \couniv{g_1}{g_T}h_{P,Q},
	\end{align*}
	where ($\alpha$) follows from the equivariance of $g_T$ and ($\beta$) holds since $(T,G,\phi, \partial)$ is a (pre)crossed module. As a result, there exists a unique morphism $g_G\colon Q\rtimes_\psi P \to G$ such that
	\begin{equation}
		g_1=g_G \couniv{s}{\partial'} \iota_1 = g_G s \qquad \text{and} \qquad \partial g_T = g_G \couniv{s}{\partial'} \iota_2 = g_G \partial'.\label{eq - def hG}
	\end{equation}

	Finally, $g_G$ is a section of $f_G$: it suffices to prove that $\couniv{s}{\partial'}=f_G g_G \couniv{s}{\partial'}$ since $\couniv{s}{\partial'}$ is a (regular) epimorphism:
	\begin{align*}
		f_G g_G \couniv{s}{\partial'} \overset{(\gamma)}{=} f_G \couniv{g_1}{g_T}
		= \couniv{f_G g_1}{f_G \partial g_T}
		\overset{(\delta)}{=}  \couniv{s}{f_G \partial g_T}
		\overset{(\varepsilon)}{=} \couniv{s}{\partial' f_T g_T}
		= \couniv{s}{\partial'}.
	\end{align*}
	Here ($\gamma$) follows from the definition of $g_G$, ($\delta$) follows from the definition of $g_1$ and ($\varepsilon$) holds since the pair $(f_T,f_G)$ is a morphism in $\XMod(\C)$.
\end{proof}

\begin{lemma}[The pair $(g_T,g_G)$ is a morphism in $\XMod(\C)$]
	The pair of morphisms $(g_T,g_G)$, where $g_T$ and $g_G$ are defined respectively in the proof of Lemma~\ref{lemma - def hT} and the proof of Lemma~\ref{lemma - def hG}, is a morphism in $\XMod(\C)$.
\end{lemma}

\begin{proof}
	By the definition of $g_G$, we already know that $\partial g_T=g_G \partial'$. Hence, it remains to prove the equivariance condition, the equality $g_T \overline{\chi} = \phi (g_G\cosmash g_T)$. To do so, according to Lemma~\ref{lemma - cosmash nice with reg epi} and Lemma~\ref{lemma - covering of A+B flat C}, it suffices to prove the equality
	\begin{multline}
		g_T \overline{\chi} \left(\couniv{s}{\partial'} \cosmash 1_Q\right)
		\counivtr{j_{P,Q,Q}}{\iota_1\cosmash 1_Q}{\iota_2 \cosmash 1_Q}\\
		= \phi (g_G\cosmash g_T) \left(\couniv{s}{\partial'}\cosmash 1_Q\right)
		\counivtr{j_{P,Q,Q}}{\iota_1\cosmash 1_Q}{\iota_2 \cosmash 1_Q}.\label{eq - equivariance of hT hG}
	\end{multline}
	Since the family $\{\iota_1,\iota_2,\iota_3\}$
	is jointly epic as the injections of a coproduct, the equality~\eqref{eq - equivariance of hT hG} holds if and only if, for $i=1,2,3$, the equations
	\begin{multline*}
		g_T \overline{\chi} \left(\couniv{s}{\partial'} \cosmash 1_Q\right)
		\counivtr{j_{P,Q,Q}}{\iota_1\cosmash 1_Q}{\iota_2 \cosmash 1_Q}\iota_i\\
		= \phi (g_G\cosmash g_T) \left(\couniv{s}{\partial'}\cosmash 1_Q\right)
		\counivtr{j_{P,Q,Q}}{\iota_1\cosmash 1_Q}{\iota_2 \cosmash 1_Q}\iota_i
	\end{multline*}
	hold.

	For the first inclusion $\iota_1$, we have the diagram
	\begin{equation}
		\vcenter{\xymatrix{
		P\cosmash Q \cosmash Q\ar[r]^-{j_{P,Q,Q}} & (P+Q)\cosmash Q \ar@{->>}[r]^-{\couniv{s}{\partial'}\cosmash 1_Q}& (Q\rtimes_\psi P)\cosmash Q\ar[r]^-{\overline{\chi}} \ar[d]_-{g_G\cosmash g_T}& Q\ar[d]^-{g_T}\\
		&& G\cosmash\ar[r]_-{\phi} T & T
		}}\label{diag - equivariance ternary cosmash}
	\end{equation}
	We can decompose the bottom part of~\eqref{diag - equivariance ternary cosmash} as
	\[
		\xymatrix@C=3pc @R=3pc{
		P\cosmash Q \cosmash Q\ar[r]^-{j_{P,Q,Q}}\ar[d]^-{g_1\cosmash \partial g_T \cosmash g_T} \ar@/_2pc/[dd]_-{g_1\cosmash g_T\cosmash g_T} & (P+Q)\cosmash Q \ar@{->>}[r]^-{\couniv{s}{\partial'}\cosmash 1_Q} \ar[dr]|{\couniv{g_1}{\partial g_T}\cosmash g_T} & (Q\rtimes_\psi P)\cosmash Q\ar[d]^-{g_G\cosmash g_T}\\
		G\cosmash G \cosmash T \ar[rr]^-{S_{2,1}^{G,T}} \ar[drr]|{\phi_{2,1}^{G,T}} & & G\cosmash T\ar[d]^-{\phi}\\
		G\cosmash T \cosmash T\ar[u]_-{1_G\cosmash \partial \cosmash 1_T}\ar[rr]_-{\phi_{1,2}^{G,T}} && T
		}
	\]
	which is commutative since
	\begin{align*}
		\phi (g_G\cosmash g_T)\left(\couniv{s}{\partial'}\cosmash 1_Q\right) j_{P,Q,Q} & \overset{(\alpha)}{=} \phi \left( \couniv{g_1}{\partial g_T})\cosmash g_T\right) j_{P,Q,Q} \\
		                                                                               & \overset{(\beta)}{=} \phi S_{2,1}^{G,T} (g_1 \cosmash \partial g_T \cosmash g_T)           \\
		                                                                               & = \phi_{2,1}^{G,T}(1_G \cosmash \partial \cosmash 1_T) (g_1\cosmash g_T \cosmash g_T)      \\
		                                                                               & = \phi_{1,2}^{G,T} (g_1\cosmash g_T \cosmash g_T),
	\end{align*}
	where ($\alpha$) follows from the definition of $g_G$ and ($\beta$) follows from Lemma~\ref{lemma - useful lemma}.

	We can decompose the top part of the diagram~\eqref{diag - equivariance ternary cosmash} as
	\[
		\resizebox{.95\textwidth}{!}{
		\xymatrix@C=3.5pc @R=3pc{
		& P\cosmash Q \cosmash Q \ar@/_4pc/[dd]_-{S_{1,2}^{P,Q}} \ar@/_4pc/[dddl]_-{g_1\cosmash g_T\cosmash g_T}\ar[r]^-{j_{P,Q,Q}} \ar[d]|{s\cosmash 1_Q \cosmash 1_Q} \ar[rd]|-{s\cosmash \partial' \cosmash 1_Q}& (P+Q)\cosmash Q \ar[r]^-{\couniv{s}{\partial'}\cosmash 1_Q} & (Q\rtimes_\psi P)\cosmash Q\ar[d]^-{\overline{\chi}} \\
		&(Q\rtimes_\psi P) \cosmash Q \cosmash Q 	\ar[r]^-{1_{(Q\rtimes_\psi P)}\cosmash \partial' \cosmash 1_Q} \ar@/_2pc/[rr]_-{\overline{\chi}_{1,2}} \ar[dr]_-{S_{1,2}^{(Q\rtimes_\psi P), Q}} & (Q\rtimes_\psi P)\cosmash (Q\rtimes_\psi P) \cosmash Q\ar[ru]|-{S_{2,1}^{(Q\rtimes_\psi P), Q}} \ar[r]^-{\overline{\chi}_{2,1}} & Q \ar[r]^-{g_T} & T \\
		&P \cosmash Q \ar[r]|{s\cosmash 1_Q} \ar[rrd]_-{g_1\cosmash g_T} & (Q\rtimes_\psi P) \cosmash Q \ar[ru]_-{\overline{\chi}} & & \\
		G \cosmash T \cosmash T\ar[rrr]_-{S_{1,2}^{G,T}}&&& G\cosmash T \ar[ruu]_-{\phi} &
		}
		}
	\]
	which is commutative since for the clockwise composition we have
	\begin{align*}
		\overline{\chi} \left(\couniv{s}{\partial'}\cosmash 1_Q \right) j_{P,Q,Q} & \overset{(\alpha)}{=} \overline{\chi} S_{2,1}^{Q\rtimes_\psi P, Q}(s\cosmash \partial' \cosmash 1_Q)                          \\
		                                                                          & = \overline{\chi}_{2,1}^{Q\rtimes_\psi P, Q} (1_{Q\rtimes_\psi P}\cosmash \partial' \cosmash 1_Q)(s\cosmash 1_Q \cosmash 1_Q) \\
		                                                                          & \overset{(\beta)}{=} \overline{\chi}_{1,2}^{Q\rtimes_\psi P, Q}(s\cosmash 1_Q\cosmash 1_Q)
		= \overline{\chi}  S_{1,2}^{Q\rtimes_\psi P, Q} (s\cosmash 1_Q \cosmash 1_Q)                                                                                                                              \\
		                                                                          & \overset{(\gamma)}{=}\overline{\chi} (s\cosmash 1_Q)S_{1,2}^{P,Q}
		= \overline{\chi} (s\cosmash 1_Q)  S_{1,2}^{P,Q}
		\overset{(\delta)}{=} \psi S_{1,2}^{P,Q}
	\end{align*}
	where ($\alpha$) follows from Lemma~\ref{lemma - useful lemma}, ($\beta$) follows from the condition ($3$) of the Definition~\ref{thm - XMod and cosmash}, ($\gamma$) follows from Lemma~\ref{lemma - second useful lemma} and ($\delta$) follows from the proof of Lemma~\ref{lemma - def hT}. On the other hand, for the counterclockwise composition, we have
	\begin{align*}
		g_T \psi S_{1,2}^{P,Q} \overset{(\alpha')}{=} \phi  (g_1\cosmash g_T) S_{1,2}^{P,Q}
		\overset{(\beta')}{=} \phi  S_{1,2}^{G,T} (g_1\cosmash g_T \cosmash g_T)
		= \phi_{1,2}^{G,T}(g_1\cosmash g_T \cosmash g_T)
	\end{align*}
	where ($\alpha'$) follows from the equivariance of $g_T$ and ($\beta'$) follows from Lemma~\ref{lemma - second useful lemma}.

	For the second inclusion $\iota_2$, we have the diagram
	\[
		\xymatrix{
		&(P+Q)\cosmash Q\ar@{->>}[d] ^-{\couniv{s}{\partial'}\cosmash 1_Q} & \\
		P\cosmash Q\ar[ru]^-{\iota_1\cosmash 1_Q} \ar[r]_-{s\cosmash 1_Q} \ar[d]_-{1_P\cosmash g_T} & (Q\rtimes_\psi P)\cosmash Q\ar[r]^-{\overline{\chi}} \ar[d]_-{g_G\cosmash g_T}& Q\ar[d]^-{g_T}\\
		P\cosmash T\ar[r]_-{g_1\cosmash 1_T}& G\cosmash\ar[r]_-{\phi} T & T
		}
	\]
	which commutes since
	\begin{align*}
		\phi(g_G\cosmash g_T)(\couniv{s}{\partial'}\cosmash 1_Q)(\iota_1\cosmash 1_Q) & =\phi (g_G\cosmash g_T)(s\cosmash 1_Q)
		\overset{(\tilde{\alpha})}{=} \phi (g_1\cosmash g_T)                                                                                                           \\
		                                                                              & \overset{(\tilde{\beta})}{=} g_T \psi
		\overset{(\tilde{\gamma})}{=} g_T \overline{\chi} (s\cosmash 1_Q)                                                                                              \\
		                                                                              & = g_T \overline{\chi} (\couniv{s}{\partial'}\cosmash 1_Q)(\iota_1\cosmash 1_Q)
	\end{align*}
	where ($\tilde{\alpha}$) follows from the definition of $g_G$, ($\tilde{\beta}$) follows from the equivariance of $g_T$ and ($\tilde{\gamma}$) follows from the proof of Lemma~\ref{lemma - def hT}.

	For the third inclusion $\iota_3$, we have
	\[
		\xymatrix{
		Q\cosmash Q \ar@/_1pc/[drr]_-{\partial g_T \cosmash g_T} \ar[r]^-{\iota_2\cosmash 1_Q}& (P+Q)\cosmash Q \ar@{->>}[r]^-{\couniv{s}{\partial'}\cosmash 1_Q} \ar[dr]|{\couniv{g_1}{\partial g_T}\cosmash g_T} &(Q\rtimes_\psi P)\cosmash Q\ar[r]^-{\overline{\chi}} \ar[d]^-{g_G\cosmash g_T}& Q\ar[d]^-{g_T}\\
		&& G\cosmash\ar[r]_-{\phi} T & T
		}
	\]
	where the outer diagram commutes. Indeed, we have
	\begin{align*}
		\phi (\partial g_T\cosmash g_T) & =\phi (\partial\cosmash 1_T)(g_T\cosmash g_T)
		\overset{(\alpha)}{=} \overline{\chi_T} (g_T\cosmash g_T)
		=g_T\overline{\chi_Q}                                                                              \\
		                                & \overset{(\beta)}{=} g_T \overline{\chi} (\partial'\cosmash 1_Q)
		= g_T\overline{\chi} \left(\couniv{s}{\partial'}\cosmash 1_Q\right)(\iota_2\cosmash 1_Q)
	\end{align*}
	where ($\alpha$) follows from the Peiffer condition for the crossed module $(T,G,\phi,\partial)$ and ($\beta$) follows from the Peiffer condition for the crossed module $(Q,Q\rtimes_\psi P,\overline{\chi},\partial')$.
\end{proof}

\subsection{Free internal crossed modules in a variety}\label{sec - free object}
In any variety of algebras, such as the category of groups, we have an explicit description of the \defn{free objects} which are the image of any set via the right adjoint of the forgetful functor ${U\colon \V\to \Set}$ (see for instance~\cite{Adamek-Rosicky-Vitale}). Furthermore, in this context, any free object is projective (e.g.~\cite{Hilton-Stammbach}). The purpose of this section is to describe explicitly the free objects in $\XMod(\V)$. Such objects always exist since it is well known that $\XMod(\V)$ is equivalent to $\Cat(\V)$, the category of internal categories in a semi-abelian variety of algebras $\V$. However, free internal categories exist as soon as $\V$ is Mal'tsev variety (and therefore they also do in any semi-abelian variety), since then $\Cat(\V)$ is also a variety of algebras~\cite{Gran-Rosicky}. This description will have one main consequence: it will help us show that the variety $\XMod(\V)$ is not a Schreier variety.

For $\XMod(\Gp)$, in \cite{Carrasco-Homology}, the authors divide the required adjunction into a composite of two consecutive adjunctions:
\begin{equation}\label{diag - adjunction}
	\vcenter{\xymatrix{
	\XMod(\Gp) \ar@<-1.1ex>[r]_-{R}^-{\perp}& \Gp \ar@<-1.1ex>[r]_-{V}^-{\perp} \ar@<-1.1ex>[l]_-{L} & \Set \ar@<-1.1ex>[l]_-{F_r}
	}}
\end{equation}
Here the functor $F_r$ is the free functor for $\Gp$, the functor $R$ sends a crossed module $(T,G, \psi, \partial)$ to the group $T\times G$, and ${L(H)\coloneq(H\flat H,H+H, \overline{\chi}, \kappa_{H,H})}$.

We assert that the construction is identical for any semi-abelian variety $\V$. Since $\V$ is a variety, we may focus on the left-hand adjunction of~\eqref{diag - adjunction}.

\begin{proposition}\label{prop - left adjoint to R}
	Let $\V$ be a semi-abelian variety. The functor
	\[L\colon H\mapsto (H\flat H,H+H, \overline{\chi}, \kappa_{H,H})\]
	is left adjoint to the functor $R\colon\XMod(\V)\to \V$ sending a crossed module $(T,G, \psi, \partial)$ to the product $T\times G$.
\end{proposition}
\begin{proof}
	Let $H$ be an object of $\V$, $(T,G,\psi, \partial)$ an internal crossed module of $\V$, and let $f\colon H \to T$ and $g\colon H \to G$ be two morphisms of $\V$. We will construct two morphisms $f_T\colon H\flat H \to T$ and $f_G\colon H+H \to G$ of $\V$ such that, in $\V$, $g=f_G\iota_1$ and $f=f_T\eta_H$ (where $\eta_H$ the component of the unit for the monad $T_H$ (see p~\pageref{eq-unit adjunction monad action})), and such that the pair $(f_T,f_G)$ is a morphism in $\XMod(\V)$.
	\[
		\xymatrix{
		&H\flat H\ar@{{ |>}->} [r]^-{\kappa_{H,H}} \ar@{-->}[dd]^-{f_T}& H+H \ar@<.55ex>[r]^-{\couniv{1_H}{0}} \ar[dd]^-{\couniv{s}{k}(g+f)}& H \ar@<.55ex>[l]^-{\iota_1} \ar[dd]^-{g} \\
		H\ar[rru]_-(0.8){\iota_2}\ar@{-->}[ru]^-{\eta_H} \ar[dr]_-{f} & &&\\
		& T\ar@{{ |>}->} [r]^-{k} \ar[dr]_-{\iota_2} & T\rtimes_\psi G \ar@<.55ex>[r]^-{p} & G  \ar@<.55ex>[l]^-{s}\ar[dl]^-{\iota_1}\\
		& & G+T\ar@{->>}[u]|(0.35){\couniv{s}{k}} &
		}
	\]

	We may notice that $g\couniv{1_H}{0}=p\couniv{s}{k}(g+f)$ and therefore there exists a unique $f_T\colon H\flat H\to T$ such that the above left-hand square commutes. In addition, we have
	\begin{equation}\label{eq - free object XMod}
		sg=\couniv{s}{k}\iota_1 g =\couniv{s}{k}(g+f)\iota_1 \text{ and } kf = \couniv{s}{k}\iota_2 f = \couniv{s}{k}(g+f)\iota_2.
	\end{equation}

	Moreover, we can set $f_G\coloneq \couniv{g}{\partial f} \colon H+H\to G$ which implies immediately that $g=f_G \iota_1$. In addition, the equality $f=f_T\eta_H$ holds since
	\[
		k f_T \eta_H=\couniv{s}{k}(g+f)\kappa_{H,H} \eta_H=\couniv{s}{k}(g+f)\iota_2=\couniv{s}{k}\iota_2 f = k f.
	\]

	We also claim that the pair $(f_T,f_G)$ is a morphism in $\XMod(\C)$. Indeed, since $(T,G,\psi, \partial)$ is an internal (pre)crossed module, there exists a unique morphism $e\colon T\rtimes_\psi G\to G$ such that $1_G=es$ and $\partial = e k$~\cite[Proposition 5.4]{HVdL}. This observation leads to the equality $\partial f_T=f_G \kappa_{H,H}$ since
	\[
		\partial f_T = ek f_T \overset{\eqref{eq - free object XMod}}{=} e \couniv{s}{k} (g+f) \kappa_{H,H} = \couniv{1_G}{\partial } (g+f)\kappa_{H,H} = \couniv{g}{\partial f}\kappa_{H,H}= f_G \kappa_{H,H}.
	\]
	\[
		\xymatrix{
		H\flat H \ar[r]^-{\kappa_{H,H}} \ar[d]_-{f_T}& H+H\ar[d]|{ \couniv{s}{k}(g+f)} \ar@/^1pc/[rd]^-{f_G} & \\
		T\ar[r]_-{k} \ar@/_2pc/[rr]_-{\partial}& T\rtimes_\psi \ar@<1.05 ex>[r]|-{p} \ar@<-1.05ex>[r]|-{e} G & G \ar[l]|-{s}
		}
	\]

	To conclude, the equivariance of the pair $(f_T,f_G)$ remains to be shown: this is the equality $f_T \overline{\chi}=\psi (f_G\cosmash f_T)$. Indeed, we have
	\begin{align*}
		k f_T \overline{\chi} & \overset{(\alpha)}{=} \couniv{s}{k}(g+f) \kappa_{H,H} \chi i_{(H+H),H\flat H}                             \\
		                      & \overset{(\beta)}{=} \couniv{s}{k}(g+f)\couniv{1_{H+H}}{\kappa_{H,H}} \kappa_{(H+H),H\flat H} i_{(H+H),H\flat H} \\
		                      & \overset{(\delta)}{=} \couniv{s}{k}(f_G+f_T) \kappa_{(H+H),H\flat H} i_{(H+H),H\flat H}
		= \couniv{s}{k}\kappa_{G,T} (f_G\flat f_T) i_{(H+H),H\flat H}                                                                     \\
		                      & = \couniv{s}{k}\kappa_{G,T} i_{G,T}(f_G\cosmash f_T) = \couniv{s}{k}h_{G,T}(f_G\cosmash f_T)
		\overset{(\gamma)}{=} k \psi (f_G\cosmash f_T)
	\end{align*}
	where:
	\begin{itemize}
		\item ($\alpha$) by definition of $f_T$ and of $\overline{\chi}$,
		\item ($\beta$) by definition of $\chi$ (cf Example~\ref{ex - kernel and XMod}) and by Lemma~\ref{lemma - equivalence clot cosmash};
		\item ($\gamma$) by the construction of the internal semi-direct product (see page~\pageref{diag - uniqueness of action for split extension});
		\item ($\delta$) follows from the equations~\eqref{eq - free object XMod} since we have
		      \[\couniv{s}{k}(f_G+f_T)= \couniv{s}{k}(g+f)\couniv{1_{H+H}}{\kappa_{H,H}}.\]
		      On the one hand,
		      \begin{align*}
			      \couniv{s}{k}(g+f)\iota_1 & = sg = sf_G \iota_1                          \\
			      \couniv{s}{k}(g+f)\iota_2 & = kf = se k f = s \partial f = s f_G \iota_2
		      \end{align*}
		      which implies that
		      \begin{align*}
			      \couniv{s}{k}(f_G+f_T)\iota_1 & = \couniv{s}{k}\iota_1 f_G= s f_G = \couniv{s}{k}(g+f) \\ &= \couniv{s}{k}(g+f) (1_{H+H}\ \kappa_{H,H}) \iota_1,
		      \end{align*}
		      and on the other hand, we have
		      \begin{align*}
			      \couniv{s}{k}(f_G+f_T)\iota_2 & = \couniv{s}{k}\iota_2 f_T= kf_T \\ &= \couniv{s}{k}(g+f)\kappa_{H,H}= \couniv{s}{k}(g+f)(1_{H+H}\ \kappa_{H,H})\iota_2. \qedhere
		      \end{align*}
	\end{itemize}
\end{proof}

The previous result gives an explicit description of the free objects in $\XMod(\V)$:

\begin{corollary}\label{cor - free object XMod(V)}
	Let $\V$ be a semi-abelian variety. In the notation of \eqref{diag - adjunction}, the composite $L F_r$ is left adjoint to the functor $VR$. Here $F_r$ is the free functor for the forgetful functor $V\colon \V \to \Set$.\noproof
\end{corollary}

All free internal crossed modules are thus of the form
\[(F_r(X)\flat F_r(X),F_r(X)+F_r(X), \overline{\chi}, \kappa_{F_r(X),F_r(X)})\]
for some $X\in \Set$. We regain the well-known characterization in the case of groups.

A consequence of this observation is, that the projective object constructed in Theorem~\ref{thm - projective in XMod(C)} is not free.

\begin{corollary}\label{cor - XMod(V) not Schreier}
	For any semi-abelian non-trivial variety $\V$, the variety $\XMod(\V)$ is not a Schreier variety.
\end{corollary}
\begin{proof}
	Consider in $\V$ the free object $P$ on one generator and the free object $X$ on two generators. By Lemma~\ref{lemma - example projective in SSEP}, we know that
	\[
		\xymatrix{
		0\ar[r] & P\flat X \ar@{{ |>}->}[r]^-{\kappa_{P,X}} & P+X \ar@{-{ >>}}@<.5ex>[r]^-{\couniv{1_P}{0}} & P\ar[r] \ar@{{ >}->}@<.5ex>[l]^-{\iota_1}  & 0
		}
	\]
	is projective in $\SSE_{P}(\V)$. Therefore, the kernel part of this split extension defines a crossed module which is projective in $\XMod(\V)$ by Theorem~\ref{thm - projective in XMod(C)}. However, by Corollary~\ref{cor - free object XMod(V)}, it is not a free object in $\XMod(\V)$.
\end{proof}

\section{The condition \P}\label{sec - condition P}
One of the prototype examples of a semi-abelian variety is the category of groups. This category has many nice properties that are shared with its category of internal categories, i.e. the crossed modules over groups. Yet, $\XMod$ does not share all the properties of $\Gp$: for instance, it is not a Schreier variety \cite{Carrasco-Homology} and it does not admit functorial fiberwise localizations~\cite{MSS}. However, for both categories, the class of projectives objects is closed under protosplit subobjects, i.e.\ $\Gp$ and $\XMod$ satisfy \defn{Condition \P}.

In a semi-abelian category, we call \P\ the statement that for each split short exact sequence
\begin{equation}
	\xymatrix{0 \ar[r] & K \ar@{{ |>}->}[r]^-k & X \ar@{-{ >>}}@<.55ex>[r]^-{f} & Y \ar[r] \ar@{{ >}->}@<.55ex>[l]^-{s} & 0}\label{eq - def condition P}
\end{equation}
if $X$ is a projective object then $K$ is projective. Being introduced in~\cite{MC-FR-TVdL}, we already know some trivial examples of this condition (such as abelian categories and Schreier varieties) and a few non-trivial examples, such as the category of Lie algebras over an unital commutative ring $\K$.

The purpose of this section is to investigate the close link between the condition \P\ for a given semi-abelian variety $\V$ and for its variety of internal crossed modules. In this section, we prove that $\V$ satisfies \P\ if and only if $\XMod(\V)$ does. As a result, if we start with a trivial or non-trivial example of a variety satisfying \P, then we obtain an infinite list of examples of varieties satisfying \P: if $\V$ does then so does $\XMod(\V)$, then also $\XMod(\XMod(\V))$, and so on.

\subsection{A condition for $\XMod(\V)$ to satisfy \P\ }

\begin{lemma}\label{lemma - lifting crossed modules}
	Let $\V$ be a semi-abelian variety and consider a morphism of short exact sequences
	\[
		\xymatrix{
		T\ar@{{ |>}->}[r]^- k \ar@{->>}[d]_-{f_T} & G \ar@{->>}[d]^-{f_G}\ar@{-{ >>}}[r]^-{\coker(k)} & \Coker(k)\ar[d]^-{h} \\
		P\ar@{{ |>}->}[r]_-{k'} & Q\ar@{-{ >>}}[r]_-{\coker(k')} & \Coker(k')
		}
	\]
	where $Q$ is projective object in $\V$ and the $f_T$, $f_G$ are regular epimorphisms.
	If $\Coker(k')$ is projective in $\V$ then there exists a section in $\XMod(\V)$ for the regular epimorphism $(f_T,f_G)$.
\end{lemma}
\begin{proof}
	By Example~\ref{ex - kernel and XMod}, it is clear that the left-hand square of the above morphism of short exact sequences defines a morphism of crossed modules.

	As a consequence of the Short Five Lemma for regular epimorphisms (see for instance~\cite{Borceux-Bourn}), since $f_T$ is a regular epimorphism, the right-hand square is a regular pushout in the sense of~\cite{Carboni-Kelly-Pedicchio,Bourn2001}. Indeed, $h$ is a regular epimorphism since so is $f_G$.

	Via the diagram
	\[
		\xymatrix{
		T\ar@{{ |>}->}[r]^- k \ar@{->>}[dd]_-{f_T} & G \ar@{-{ >>}}[rr]^-{\coker(k)} \ar@{->>}[dd]_-{f_G} \ar@{->>}[dr]^ -{u} & & \Coker(k) \ar@{->>}[dd]_-{h} \\
		& & Q\times_{\Coker(k')} \Coker(k) \pullback \ar@{->>}[ru]|{p_2} \ar@{->>}[dl]|{p_1} & \\
		P \ar@/_1pc/@{-->}[uu]|{g_T} \ar@{{ |>}->}[r]_- {k'} & Q \ar@/_1pc/@{-->}[uu]|{g_G} \ar@/_1pc/@{-->}[ur]|{j_Q} \ar@{-{ >>}}[rr]_-{\coker(k')} & & \Coker(k') \ar@/_1pc/@{-->}[uu]_-{j_Z}
		}
	\]
	where the induced comparison morphism $u$ is a regular epimorphism, we can define a section $j_Z$ of the morphism $h$. Via the pullback, we can define $j_Q$ (a section of $p_1$) which can be lifted over $u$ to the morphism $g_G$. Finally, such a morphism induces the morphism $g_T$ which turns out to be a section of $f_T$. This defines a morphism of crossed modules (see Example~\ref{ex - kernel and XMod}).
\end{proof}

Combining the results of the previous sections, we can prove that $\XMod(\V)$ satisfies \P\ as soon as $\V$ does. For the proof, we will use another characterization of \P, valid in the case of pointed Mal'tsev varieties of algebras:
\begin{proposition}\label{propo_equivalent_conditions}\cite[Proposition 3.3.1.]{MC-FR-TVdL}
	Let $\W$ be a pointed Mal'tsev variety of algebras, with forgetful functor $U\colon \W\to \Ens$ and its left adjoint $F_r\colon \Ens\to \W$. The following conditions are equivalent:
	\begin{tfae}
		\item the condition \P\ holds in $\W$;
		\item for any split epimorphism $g\colon A \to B$ in $\Ens$, the kernel of $F_r(g)$ is projective in $\W$.\noproof
	\end{tfae}
\end{proposition}
The assumption of the previous result holds since it is well known that $\XMod(\V)$ is a semi-abelian variety as soon as $\V$ is a semi-abelian variety~\cite{Gran-Rosicky,Bourn-Gran}.

Let $f\colon X\to Y$ a split epimorphism in $\Set$ with a section $s$. By the construction of the free object in $\XMod(\V)$ (see Section~\ref{sec - free object}), we have
\begin{equation}
	\resizebox{.9\textwidth}{!}{$\vcenter{\xymatrix@C=4.5pc{
		& 0\ar@{-->}[d] & 0\ar[d] & 0\ar[d] & \\
		0\ar[r] & P \ar@{{ |>}->}[r] \ar@{-->}[d] & F_r(X)\flat F_r(X)\ar@{-{ >>}}@<.55ex>[r]^-{F_r(f)\flat F_r(f)} \ar@{{ |>}->}[d] & F_r(Y)\flat F_r(Y)\ar[r] \ar@{{ >}->}@<.55ex>[l]^-{F_r(s)\flat F_r(s)} \ar@{{ |>}->}[d] & 0 \\
		0\ar[r] & Q \ar@{{ |>}->}[r] \ar@{-->}@<.55ex>[d] & F_r(X)+F_r(X) \ar@{-{ >>}}@<.55ex>[r]^-{F_r(f)+F_r(f)} \ar@<.55ex>@{-{ >>}}[d]^-{\couniv{1_{F_r(X)}}{0}} & F_r(Y)+F_r(Y) \ar[r] \ar@{{ >}->}@<.55ex>[l]^-{F_r(s)+F_r(s)} \ar@<.55ex>@{-{ >>}}[d]^-{\couniv{1_{F_r(X)}}{0}} & 0\\
		0\ar[r] & Z \ar@{{ |>}->}[r] \ar@{-->}@<.55ex>[u] \ar@{-->}[d] & F_r(X) \ar@{{ >}->}@<.55ex>[u]^-{\iota_1} \ar@{-{ >>}}@<.55ex>[r]^-{F_r(f)} \ar[d]  & F_r(Y)\ar@{{ >}->}@<.55ex>[u]^-{\iota_1}\ar[r] \ar@{{ >}->}@<.55ex>[l]^-{F_r(s)} \ar[d]  & 0 \\
		& 0 & 0 & 0 &
		}}$}\label{diagm: P in XMod}
\end{equation}
where
\begin{itemize}
	\item all the horizontal sequences are (split) short exact sequences by construction;
	\item the middle vertical sequence and the right vertical sequence are, respectively, the free object on the set $X$ and $Y$, and there are (split) short exact sequences by construction;
	\item the morphisms in the left-hand vertical sequence are restrictions to the kernels;
	\item the free objects $F_r(X)$ and $F_r(Y)$ are projective in $\V$, and therefore ${F_r(X)+F_r(X)}$ and $F_r(Y)+F_r(Y)$ are projective as well;
	\item since $\V$ satisfies \P, we have respectively that
	      \begin{itemize}
		      \item the kernels $F_r(X)\flat F_r(X)$ and $F_r(Y)\flat F_r(Y)$ of the solid vertical split short exact sequences are also projective;
		      \item the object $Z\coloneq \ker(F_r(f))$ is projective in $\V$;
		      \item the object $Q\coloneq \ker(F_r(f)+F_r(f))$ is projective in $\V$;
		      \item the object $P\coloneq \ker(F_r(f)\flat F_r(f))$ is projective in $\V$;
	      \end{itemize}
	\item the left-hand vertical sequence with the dashed arrows is also a (split) short exact sequence since the (kernel) functor preserves split epimorphisms and kernels commute with kernels.
\end{itemize}

We will prove that the internal crossed module $(P,Q,\overline{\chi},k)$ is projective.

\begin{theorem}\label{thm - proof of P XMod(V)}
	If the semi-abelian variety $\V$ satisfies the condition \P, then does the category $\XMod(\V)$.
\end{theorem}

\begin{proof}
	The crossed module $(P,Q,\overline{\chi}, k)$ defined in~\eqref{diagm: P in XMod} is a retract of a projective crossed module constructed via Theorem~\ref{thm - projective in XMod(C)}: by Lemma~\ref{lemma - H-group enough projectives}, we can cover the split extension $P$ of $Z$ by a projective split extension of $Z$ as in
	\[
		\xymatrix{
		0 \ar[r] & R_P \ar@{->>}[d]_{f_P} \ar@{{ |>}->}[r]^-{k'} & R_P\rtimes Z \ar@{->>}[d]_{f_Q} \ar@{-{ >>}}@<.55ex>[r]^-{f'} & Z \ar@{=}[d] \ar[r] \ar@{{ >}->}@<.55ex>[l]^-{s'} & 0\\
		0 \ar[r] & P \ar@{{ |>}->}[r]^-{k} & Q \ar@{-{ >>}}@<.55ex>[r]^-{g} & Z \ar[r] \ar@{{ >}->}@<.55ex>[l]^-{t} & 0
		}
	\]

	Finally, we can conclude via Lemma~\ref{lemma - lifting crossed modules} that $(P,Q,\overline{\chi}, k)$ is a retract of the crossed module $(R_P, R_P\rtimes Z, \overline{\chi'}, k')$ which is projective by Proposition~\ref{thm - projective in XMod(C)}.
\end{proof}

\subsection{Some results concerning $\pi_0$}

Consider a semi-abelian variety $\V$ and the \defn{connected components functor} $\pi_0\colon \XMod(\V)\to \V$ defined on objects as the coequalizer of the morphisms $d$ and $c$ coming from the underlying reflexive graph structure
\[
	\xymatrix{0 \ar[r] & T \ar@{{ |>}->}[r]^-k & T\rtimes_{\psi} G \ar@{-{ >>}}@<1ex>[r]^-{d}\ar@<-1ex>[r]_-{c} & G \ar[r] \ar@{{ >}->}[l]|-{e} & 0}
\]
where $(T,G,\psi,\partial)\in \XMod(\V)$. We can also define $\pi_0$ equivalently~\cite[Proposition 3.9]{EverVdL2} as the cokernel of $ck=\partial$.

Since the category is semi-abelian, the functor is protoadditive~\cite{EG-honfg}. In addition, it is the left adjoint to the \defn{discrete crossed module functor} $D\colon \V\to \XMod(\V)$ which sends an object $X$ to the (discrete) crossed module $0\to X$ with the trivial action. As a result, $\pi_0$ preserves all colimits. We can also prove that it preserves \defn{proper morphisms} (morphisms with cokernel-kernel factorization):

\begin{lemma}\label{lemma - reflector sequ right exact}
	If $\V$ is a semi-abelian variety and $F\colon {\V\to \W}$ is the reflector from $\V$ to a subvariety $\W$, then $F$ is sequentially right-exact~\cite{PVdL1}. In particular, $F$ preserves proper morphisms.
\end{lemma}
\begin{proof}
	To be a sequentially right-exact functor, we need to prove that the image of an exact sequence in $\V$
	\[\xymatrix{K \ar[r]^-{k} & A \ar@{-{ >>}}[r]^-{f} & B \ar[r] & 0,}
	\]
	is the exact sequence
	\[
		\xymatrix{F(K) \ar[r]^-{F(k)} & F(A) \ar@{-{ >>}}[r]^-{F(f)} & F(B) \ar[r] & 0.}
	\]

	Since $F$ is a left adjoint functor, it preserves cokernels.

	Therefore, to conclude, it suffices to prove that $F$ preserves proper morphisms as well. Let $g\colon K\to A$ a proper morphism and consider $U\colon {\W\to \V}$ the exact right adjoint to $F$
	\[
		\xymatrix{K \ar@{-{ >>}}[r]_-{c} \ar@{-{ >>}}[d]_-{\eta_K} \ar@/^1.5pc/[rr]^-{g} & \mathrm{Im}(g)\ar@{-{ >>}}[d]|{\eta_{\mathrm{Im}(g)}} \ar@{{ |>}->}[r]_-{k} & A \ar@{-{ >>}}[d]^-{\eta_A} \\
		UF(K) \ar@{-{ >>}}[r]^-{UF(c)} \ar@/_1.5pc/[rr]_-{UF(g)}&UF(\mathrm{Im}(g)) \ar[r]^-{UF(k)} & UF(A) }
	\]
	Since the components of the unit $\eta$ of the adjunction are regular epimorphisms~\cite{Janelidze-Kelly}, and since regular images of kernels are kernels in the semi-abelian $\V$~\cite{Janelidze-Marki-Tholen}, this implies that the monomorphism part of the regular image factorization of $UF(k)$ is a kernel and therefore that the morphism $UF(k)$ is proper.

	Finally, since $\W$ is closed under subobjects, the right adjoint $U$ reflects proper morphisms.
\end{proof}

\begin{example}
	If we consider a semi-abelian variety $\V$, then $\pi_0\colon \XMod(\V)\to \V$ is a reflector from the semi-abelian variety $\XMod(\V)$ to the Birkhoff subcategory $\V$. As a result, by Lemma~\ref{lemma - reflector sequ right exact}, $\pi_0$ preserves proper morphisms.
\end{example}

\subsection{When $\XMod(\V)$ satisfies \P}
For the moment, we only know that if a semi-abelian variety $\V$ satisfies \P, then $\XMod(\V)$ also does (see Theorem~\ref{thm - proof of P XMod(V)}). In this section, we will investigate the other implication.
\begin{lemma}\label{lemma - pi0 preserves projective}
	The functor $\pi_0$ preserves the class of projective objects. As a result, if a discrete crossed module $0\to X$ (for $X\in \C$) is projective, then $X$ is projective in $\C$.
\end{lemma}
\begin{proof}
	Since the functor $D$ preserves regular epimorphisms, $\pi_0$ preserves projective objects (see for instance~\cite{Hilton-Stammbach}). However, $\pi_0(0\to X)={\Coker(0\colon 0\to X)}=X$.
\end{proof}
\begin{lemma}\label{lemma - projective in C implies projective in XMod}
	Consider a projective object $X$ in a semi-abelian category $\C$. Then the discrete crossed module $0\to X$ is projective in $\XMod(\C)$.
\end{lemma}
\begin{proof}
	Since $0$ is the free object on the empty set and since the discrete crossed module $0\to X$ can be rewritten as the kernel part of
	\[
		\xymatrix{
		0\ar[r] & X\flat 0\cong 0 \ar@{{ |>}->}[r]^-{\kappa_{X,0}} & X+0\cong X \ar@{-{ >>}}@<.5ex>[r]^-{\couniv{1_X}{0}} & X\ar[r] \ar@{{ >}->}@<.5ex>[l]^-{\iota_1}  & 0
		}
	\]
	the result follows from Lemma~\ref{lemma - example projective in SSEP} and Theorem~\ref{thm - projective in XMod(C)}.
\end{proof}
\begin{theorem}\label{thm - P iff}
	If we consider a semi-abelian variety $\V$, then $\V$ satisfies \P\ if and only if $\XMod(\V)$ satisfies \P.
\end{theorem}
\begin{proof}
	By Theorem~\ref{thm - proof of P XMod(V)}, it suffices to prove the necessary condition. Consider a split short exact sequence as~\eqref{eq - def condition P} where $X$ is projective in $\V$. This can be extended to a split short exact sequence in $\XMod(\V)$ as
	\[
		\xymatrix{
		0 \ar[r]& 0\ar[d] \ar@{=}[r] & 0 \ar@{=}[r]\ar[d]& 0\ar[d] \ar[r]&0\\
		0 \ar[r] & K \ar@{{ |>}->}[r]^-k & X \ar@{-{ >>}}@<.55ex>[r]^-{f} & Y \ar[r] \ar@{{ >}->}@<.55ex>[l]^-{s} & 0
		}
	\]
	Since $X$ is projective in $\V$, then by Lemma~\ref{lemma - projective in C implies projective in XMod}, the discrete crossed module $0\to X$ is projective in $\XMod(\V)$. As a result, $0\to K$ is also projective by the condition \P. Finally, $K$ is projective by Lemma~\ref{lemma - pi0 preserves projective}.
\end{proof}

\section{Non-additive derived functors of \texorpdfstring{$\pi_0$}{Pi0}}\label{sec - non additive functor pi0}

The previous sections explain what are the assumptions required on a given semi-abelian variety of algebras $\V$ to guarantee the computation of the non-additive derived functors (in the sense of \cite{MC-FR-TVdL}) for $\pi_0$.

\begin{theorem}\label{thm - derived functor XMod(V)}
	Let $\V$ be a semi-abelian variety where the condition \P\ holds. If we consider the protoadditive functor $\pi_0\colon \XMod(\V) \to \V$, then we can compute the non-additive derived functor $\Left_n(\pi_0)$ for any $n\in \Z$ and for any object in $\XMod(\V)$.
\end{theorem}
\begin{proof}
	The proof is a consequence of the previous sections and a theorem in~\cite{MC-FR-TVdL}.
\end{proof}

\subsection{Concrete examples and direct applications}
In this section, we will discuss some varieties satisfying the assumptions of Theorem~\ref{thm - derived functor XMod(V)}.

A first non-abelian variety is the category of $\Gp$ where the condition \P\ is automatically satisfied since it is a Schreier variety~\cite{Nielsen,Schreier}.

\begin{corollary}[Internal crossed modules over $\Gp$~\cite{MC-FR-TVdL}]
	If we consider the functor $\pi_0\colon \XMod(\Gp) \to \Gp$ then we can compute the non-additive derived functors $\Left_n(\pi_0)\colon \XMod(\Gp)\to \Gp$ for any $n\in \Z$ and  for any object in $\XMod(\Gp)$.\noproof
\end{corollary}

A first non-abelian and non-Schreier variety is the category of Lie algebras $\LieK$ over an unital commutative ring $\K$ (which is not a Schreier variety (see e.g.\ \cite{Bahturin}). However, this category satisfies the condition \P~\cite{MC-FR-TVdL}.

\begin{corollary}[Internal crossed modules over $\LieK$]
	If we consider the functor $\pi_0\colon \XMod(\LieK) \to \LieK$ then we can compute the non-additive derived functors $\Left_n(\pi_0)\colon \XMod(\LieK)\to \LieK$ for any $n\in \Z$ and  for any object in $\XMod(\LieK)$.\noproof
\end{corollary}

\subsection{Iterated applications and new examples}

Another kind of example comes from iterations of Theorem~\ref{thm - derived functor XMod(V)} where we use the same underlying category~\cite{EG-honfg}. Indeed, if we start with a semi-abelian variety $\V$ that satisfies \P, then we have the iterated adjunctions
\[
	\xymatrix{
	\V \ar@<-1.1ex>[r]_-{D}^-{\perp}& \XMod(\V) \ar@<-1.1ex>[r]_-{D}^-{\perp} \ar@<-1.1ex>[l]_-{\pi_{0,\V}} & \XMod^2(\V) \ar@<-1.1ex>[l]_-{\pi_{0,\XMod(\V)}}
	}
\]
where $\XMod^2(\V)\coloneq \XMod(\XMod(\V))$. Let's denote the above composition as $\pi_{0,\V}^2\coloneq \pi_{0,\V}\pi_{0,\XMod(\V)}$ and by extension $\pi_{0,\V}^n\coloneq \pi_{0,\V}\pi_{0,\V}^{n-1}$ for $n\geq 1$.
It turns out that the assumptions of Theorem~\ref{thm - derived functor XMod(V)} are satisfied for each adjunction separately and also by the composition.

\begin{corollary}[General construction]
	Let $\V$ be a semi-abelian variety that satisfies the condition \P. If we consider the compositions
	\[
		\pi_{0,\XMod^n(\V)}^m\colon \XMod^{m+n}(\V)\to \XMod^{n}(\V)
	\]
	of $\pi_0$ for $m,n\in \N$ such that $m>1$, then we can compute the non-abelian derived functor $\Left_k(\pi_{0,\XMod^n(\V)}^m)(X)$ for any $k\in\Z$ and for any object $X$ in $\XMod^{m+n}(\V)$.

	In particular, the same is true for the composition $\pi_{0,\V}^m\colon \XMod^m(\V)\to \V$.\noproof
\end{corollary}

\section*{Acknowledgments}
Many thanks to my supervisor, Tim Van der Linden, who introduced me to the theory of internal crossed modules and the theory of cosmash products. His feedback throughout the different stages of this article has been truly priceless.


\end{document}